\documentclass[letterpaper, 10 pt, conference]{ieeeconf}  % Comment this line out if you need a4paper
\usepackage{gen_settings}

\IEEEoverridecommandlockouts                              % This command is only needed if 
\title{\LARGE \bf
Bounding Optimality Gaps for Non-Convex Optimization Problems: \\ Applications to Nonlinear Safety-Critical Systems
}

\author{Prithvi Akella and Aaron D. Ames$^{1}$% <-this % stops a space
\thanks{*This work was supported by the AFOSR Test and Evaluation Program, grant FA9550-19-1-0302}% <-this % stops a space
\thanks{$^{1}$All authors are with the California Institute of Technology
        {\tt\small \{pakella, ames\}@caltech.edu}}%
}

\begin{document}

\maketitle
\thispagestyle{empty}
\pagestyle{empty}

%%%%%%%%%%%%%%%%%%%%%%%%%%%%%%%%%%%%%%%%%%%%%%%%%%%%%%%%%%%%%%%%%%%%%%%%%%%%%%%%
\begin{abstract}
Efficient methods to provide sub-optimal solutions to non-convex optimization problems with knowledge of the solution's sub-optimality would facilitate the widespread application of nonlinear optimal control algorithms.  To that end, leveraging recent work in risk-aware verification, we provide two algorithms to (1) probabilistically bound the optimality gaps of solutions reported by novel percentile optimization techniques, and (2) probabilistically bound the maximum optimality gap reported by percentile approaches for repetitive applications, \textit{e.g.} Model Predictive Control (MPC).  Notably, our results work for a large class of optimization problems.  We showcase the efficacy and repeatability of our results on a few, benchmark non-convex optimization problems and the utility of our results for controls in a Nonlinear MPC setting.

\end{abstract}

%%%%%%%%%%%%%%%%%%%%%%%%%%%%%%%%%%%%%%%%%%%%%%%%%%%%%%%%%%%%%%%%%%%%%%%%%%%%%%%%
\section{INTRODUCTION} 
Optimal controllers have emerged in recent years as the preeminent choice for controller synthesis as they offer a natural way of expressing (perhaps) disparate control objectives~\cite{lewis2012optimal,kalman1960contributions,locatelli2002optimal}, \textit{e.g.} as in Model Predictive Control (MPC)~\cite{camacho2013model,rawlings2000tutorial,garcia1989model}, control-barrier-function based quadratic programs~\cite{ames2016control,xu2015robustness,grandia2021multi}, and optimal path planning~\cite{raja2012optimal,noreen2016optimal,riviere2020glas}, among others.  Specifying to the first two examples, optimal controllers in this vein leverage existing system models to inform controller choice at the current or successive time-steps, \textit{e.g.} the predictive aspect of MPC references the utilization of a system model to predict how a given input will affect system evolution and uses that prediction to determine optimal input choice for the next few time-steps.  Notably, one-step control-barrier-function-based quadratic programs ensure convexity for nonlinear system models, though predicting over these models remains a non-convex problem~\cite{ma2021model,agrawal2017discrete,zeng2021enhancing,zeng2021safety}.

The utility of these controllers inspired efforts to develop and exploit rapid solution techniques for these non-convex optimization problems, though some improvement avenues still exist.  For example, typical approaches in the literature include Model Predictive Path Integral Control (MPPI)~\cite{williams2017model,gandhi2021robust}, Collocation methods~\cite{tamimi2009nonlinear,tamimi2010combined}, Bayesian Optimization~\cite{andersson2016model,jain2020computing}, Learning-based methods~\cite{vaupel2020accelerating,limon2017learning}, and more recently, percentile-based methods~\cite{akella2022sample,akella2023probabilistic}.  In specific for the latter three learning-based methods, both bayesian optimization and other learning-based methods guarantee eventual convergence to an optimal solution but do not provide sample complexity bounds.  Percentile methods offer robust sample complexity bounds and have seen considerable success at efficiently providing sub-optimal solutions to constrained nonlinear optimal controllers.  However, unlike the first two methods, they fail to provide even probabilistic bounds on the optimality gap between the reported percentile solution and the true optimal value.

\begin{figure}[t]
    \centering
    \includegraphics[width = \columnwidth]{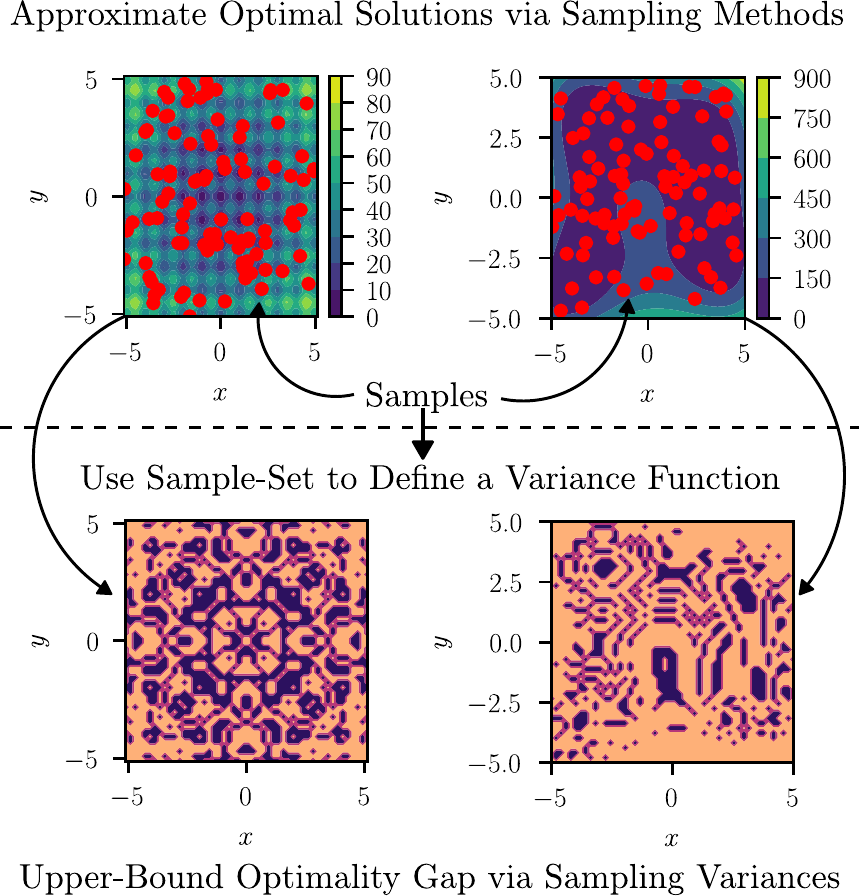}
    \caption{An overview of our proposed method.  Sampling methods employed in risk-aware verification provide decent results to non-convex optimization problems, though their optimality gaps are not well-understood.  By defining a variance function around the previously generated sample set, we can implicitly identify decisions (orange on the bottom) whose variances exceed the optimality gap of our reported sampling solution and identify such an optimality gap via a second application of the same sampling methods.} \vspace{-0.2 in}
    \label{fig:title}
\end{figure}

\newidea{Our Contribution:} We aim to provide probabilistic guarantees on the optimality gap of percentile solutions --- the deviance in cost evaluation of the percentile solution and the minimum function value --- for their respective optimization problems.  In doing so, we hope to develop the first in a line of algorithms that provide robust sample complexity bounds for efficient identification of sub-optimal solutions with knowledge of the solution's sub-optimality.  To that end, our contributions are twofold:
\begin{itemize}
\item Subject to a fairness assumption, we provide sample-complexity bounds for the identification of the optimality gap for percentile solutions to a large class of optimization problems.
\item For optimization problems to be solved repeatedly, \textit{e.g.} MPC, we provide a method to determine the probability of percentile solutions exhibiting optimality gaps lower than a determined upper bound.
\end{itemize}
We develop such a procedure to offer efficient solutions to constrained optimal control problems with known sample complexity and sub-optimality bounds.  This procedure would deviate from existing approaches insofar as it would consider general (perhaps) nonlinear constraints for optimal controllers, and provide sample-complexity results on the sub-optimality of the provided solution --- a discrepancy with existing methods, \textit{e.g.} MPPI, Collocation, Bayesian Optimization, and other, learning-based approaches.

\newidea{Structure:} Section~\ref{sec:percentile_review} reviews percentile optimization.  Section~\ref{sec:bounding_gap} outlines our solution to bound the optimality gaps of percentile solutions.  Section~\ref{sec:repetitive_gap} extends these results for optimization problems to be solved repeatedly, \textit{e.g.} Nonlinear MPC.  Finally, Section~\ref{sec:examples} provides examples to validate our theoretical results. 
 Specifically, Section~\ref{sec:tsp} validates the probabilistic bounding results of Section~\ref{sec:bounding_gap}, and Section~\ref{sec:robotarium} illustrates the utility of the results in Section~\ref{sec:repetitive_gap} in a Nonlinear MPC setting.

\section{A Review of Percentile Optimization}
\label{sec:percentile_review}
To facilitate the statement of our contributions, we first review percentile optimization in this section.  Consider the following optimization problem:
\begin{equation}
    \label{eq:general_opt}
    J^* = \min_{\decisionvariable \in \decisionspace}~J(\decisionvariable),
\end{equation}
subject to the following assumption:
\begin{assumption}
\label{assump:boundedness}
The decision space $\decisionspace$ is a set with bounded volume, \textit{i.e.} $\int_{\decisionspace}~1~ds = V_{\decisionspace} < \infty$ or $\decisionspace$ has a finite number of elements.  Furthermore, the cost function $J$ is bounded over $\decisionspace$, \textit{i.e.} $\exists~m,M \in \mathbb{R}, \suchthat m\leq J(\decisionvariable) \leq M,~\forall~\decisionvariable \in \decisionspace$.
\end{assumption}
This assumption permits us to define two functions, $\mathcal{V}$, which corresponds to the volume fraction occupied by a subset $A$ of $\decisionspace$, and $F$, which corresponds to the set of strictly better decisions for a provided decision $\decisionvariable' \in \decisionspace$ (we assume continuous spaces for notational ease):
\begin{gather}
    \label{eq:volume_fraction}
    \mathcal{V}(A) = \frac{\int_A~1~ds}{\int_{\decisionspace}~1~ds}, \\
    \label{eq:falsifying_set}
    F(\decisionvariable') = \{\decisionvariable \in \decisionspace~|~J(\decisionvariable) < J(\decisionvariable')\}.
\end{gather}
Naturally then, for a given decision $\decisionvariable' \in \decisionspace$, were $\mathcal{V}(F(\decisionvariable')) \leq \epsilon$ for some $\epsilon \in (0,1]$, \textit{i.e.} $\decisionvariable'$ is such that the volume fraction of strictly better decisions is no more than $\epsilon$, then $\decisionvariable'$ would be in the $100(1-\epsilon)\%$-ile with respect to minimizing $J$.  Likewise, the associated minimum cost of such a decision $J(\decisionvariable')$ should also be a probabilistic lower bound on achievable costs.  Both of these notions are expressed formally in the theorem below, which combines similar results from~\cite{akella2022sample,akella2022scenario}.  To facilitate the statement of results to follow, we will first state the assumption, then the theorem.
\begin{assumption}
    \label{assump:percentile}
    Let $I_1=\{(\decisionvariable_i,J(\decisionvariable_i))\}_{i=1}^{N_p}$ be a set of $N_p$ decisions and costs for decisions $\decisionvariable_i$ sampled independently via $\uniform[\decisionspace]$, with $\zeta^*_{N_p}$ the minimum sampled cost and $\decisionvariable^*_{N_p}$ the (perhaps) non-unique decision with minimum cost.
\end{assumption}
\begin{theorem}
\label{thm:prob_optimality}
Let Assumptions~\ref{assump:boundedness} and~\ref{assump:percentile} hold. Then $\forall~\epsilon \in [0,1]$, the probability of sampling a decision whose cost is at-least $\zeta^*_{N_p}$ is at minimum $1-\epsilon$ with confidence $1-(1-\epsilon)^{N_p}$, \textit{i.e.}
\begin{equation}
    \label{eq:prob_ver}
    \prob^{N_p}_{\uniform[\decisionspace]}
    \left[\prob_{\uniform[\decisionspace]}\left[J(\decisionvariable) \geq \zeta^*_{N_p}\right] \geq 1-\epsilon \right] \geq 1-(1-\epsilon)^{N_p}.
\end{equation}
Furthermore, $\forall~\epsilon \in (0,1]$, $\decisionvariable^*_{N_p}$ is in the $100(1-\epsilon)\%$-ile with minimum confidence $1-(1-\epsilon)^{N_p}$, \textit{i.e.}
\begin{equation}
    \label{eq:percent_opt}
    \prob^{N_p}_{\uniform[\decisionspace]}\left[\mathcal{V}(F(\decisionvariable^*_{N_p})) \leq \epsilon \right] \geq 1-(1-\epsilon)^{N_p}.
\end{equation}
\end{theorem}

\section{Bounding the Optimality Gap}
To formally bound the optimality gap, we must first define it.  To that end, consider the same general optimization problem as per~\eqref{eq:general_opt}. The optimality gap is defined as follows.
\begin{definition}
\label{def:optimalitygap}
For general optimization problems of the form in~\eqref{eq:general_opt}, the \textit{optimality gap} of a decision $\decisionvariable \in \decisionspace$, denoted as $\optimalitygap(\decisionvariable)$, is the deviance between the decision and optimal values, \textit{i.e.} $\optimalitygap(\decisionvariable) = J(\decisionvariable) - J^*$.
\end{definition}
\noindent Then our formal problem statement follows.
\begin{problem}
Let Assumption~\ref{assump:percentile} hold, and let the optimality gap $\optimalitygap$ be as per Definition~\ref{def:optimalitygap}.  Identify an upper bound to $\optimalitygap(\decisionvariable^*_{N_p})$ and the probability with which this bound holds. 
\end{problem}

\subsection{Results for Bounded Optimization Problems}
\label{sec:bounding_gap}
First, we note that if Assumption~\ref{assump:percentile} holds, we can define the following variance function $\variance$ over the set $I_1$:
\begin{equation}
    \label{eq:variance}
    \variance: \decisionspace \to \mathbb{R} \suchthat \variance(\decisionvariable) = \min_{\decisionvariable_i \in D \subseteq I_1}~|J(\decisionvariable) - J(\decisionvariable_i)|.
\end{equation}
The restriction of $\decisionvariable \in D \subseteq I_1$ in the definition of $\variance$ above is purely for practical purposes.  From a theoretical standpoint, one could use the entire information set $I_1$ to define $\mathbb{V}$, but this tends to increase sample requirements as will be discussed in sections to follow (\textit{e.g.} Figure~\ref{fig:changing_V})

Intuitively then, we aim to maximize $\variance$ via a percentile method to identify a variance that supersedes the optimality gap $\optimalitygap(\decisionvariable^*_{N_p})$ of our chosen decision.  To do so, we first require the following fairness assumption --- that it is possible to sample variances at least as large as the optimality gap, as otherwise, it would be impossible to take a percentile approach.  To formally state this assumption, we first define $\Omega_r$ to be the $r$-level set of $\variance$:
\begin{equation}
    \label{eq:level_set}
    \Omega_r = \{\decisionvariable \in \decisionspace~|~\variance(s) \leq r\}.
\end{equation}
\begin{assumption}
    \label{assump:fairness}
    Let the variance function $\variance$ be as per~\eqref{eq:variance}, let the optimality gap $\optimalitygap$ be as per Definition~\ref{def:optimalitygap}, let $\mathcal{V}$ be as per~\eqref{eq:volume_fraction}, let $\Omega_r$ be as per~\eqref{eq:level_set}, and let Assumption~\ref{assump:percentile} hold.  The level set $\Omega_{\optimalitygap(s^*_{N_p})}$ of decisions whose variance is at most the optimality gap of $\decisionvariable^*_{N_p}$ does not encompass $\decisionspace$, \textit{i.e.}
    \begin{equation}
    \label{eq:fairness}
    \mathcal{V}\left(\decisionspace \setminus \Omega_{\optimalitygap(\decisionvariable^*_{N_p})}\right) > 0.
    \end{equation}
\end{assumption}

\noindent Second, we have the following result regarding level sets $\Omega$.
\begin{lemma}
\label{lem:equivalencies}
Let $\Omega_r$ be as defined in~\eqref{eq:level_set} and let $\mathcal{V}$ be as per~\eqref{eq:volume_fraction}.  The following statements are all equivalent.
\begin{enumerate}[label=(\alph*)]
    \item $r \geq s$,
    \item $\Omega_r \supseteq \Omega_s$,
    \item $\mathcal{V}(\Omega_r) \geq \mathcal{V}(\Omega_s)$,
    \item $\prob_{\uniform[\decisionspace]}[\Omega_r] \geq \prob_{\uniform[\decisionspace]}[\Omega_s]$.
\end{enumerate}
\end{lemma}
\begin{proof}
The equivalency between (a) and (b) stems via the definition of $\Omega_r$ in~\eqref{eq:level_set}.  The equivalency between (b) and (c) stems via the definition of the volume fraction function in~\eqref{eq:volume_fraction}.  Finally, the equivalency between (c) and (d) stems via the uniform distribution assigning probabilistic weight to subsets of $A\subseteq \decisionspace$ equivalent to $\mathcal{V}(A)$.
\end{proof}

\noindent Third, based on our fairness assumption, we know there exists a non-zero probability of sampling decisions such that their variances are at least the optimality gap.
\begin{lemma}
    \label{lem:min_prob}
    Let Assumption~\ref{assump:fairness} hold, then there exists $p > 0$ corresponding to the probability of uniformly sampling a decision $\decisionvariable \in \decisionspace \setminus \Omega_{\optimalitygap(\decisionvariable^*_{N_p})}$, \textit{i.e.}
    \begin{equation}
        \label{eq:min_prob}
        \prob_{\uniform[\decisionspace]}\left[\decisionspace \setminus \Omega_{\optimalitygap(\decisionvariable^*_{N_p})}\right] = p > 0.
    \end{equation}
\end{lemma}
\begin{proof}
    The result holds by definition of the uniform distribution over $\decisionspace$ and equation~\eqref{eq:fairness}.
\end{proof}

Then the main result in the utilization of a percentile approach to upper bound the optimality gap stems from the prior lemmas and assumption.  Formally, we aim to take a percentile approach to the following optimization problem:
\begin{equation}
    \label{eq:variance_opt}
    \max_{\decisionvariable \in \decisionspace}~\variance(\decisionvariable),
\end{equation}
which results in the following theorem.
\begin{theorem}
    \label{thm:bound_gap}
    Let Assumptions~\ref{assump:boundedness} and~\ref{assump:fairness} hold, let $p$ satisfy~\eqref{eq:min_prob}, let $I_2 = \{(\decisionvariable_i,\variance(\decisionvariable_i))\}_{i=1}^{N_v}$ be a set of $N_v$ decisions $\decisionvariable_i$ independently sampled from $\uniform[\decisionspace]$ with their corresponding variances as per~\eqref{eq:variance} and with $\variance^*_{N_v}$ the maximum sampled variance.  Then, $\forall \epsilon \in [0,p]$, $\variance^*_{N_v}$ exceeds the optimality gap of the percentile solution with confidence $1-(1-\epsilon)^{N_v}$, \textit{i.e.}
    \begin{equation}
        \label{eq:prob_bound_gap}
        \prob^{N_v}_{\uniform[\decisionspace]}\left[\variance^*_{N_v} \geq \optimalitygap(\decisionvariable^*_{N_p}) \right] \geq 1-(1-\epsilon)^{N_v}.
    \end{equation}
\end{theorem}
\begin{proof}
    First, we know there exists a $p > 0$ satisfying~\eqref{eq:min_prob} via Lemma~\ref{lem:min_prob}.  Second, we know that for the optimization problem~\eqref{eq:variance_opt}, the decision space $\decisionspace$ and objective function $\variance$ are bounded --- this stems via Assumption~\ref{assump:boundedness}.  This permits us to take a percentile approach to solving~\eqref{eq:variance_opt}.  Via Theorem~\ref{thm:prob_optimality} and looking at~\eqref{eq:variance_opt} as a minimization, we have the following result substituting terms for~\eqref{eq:prob_ver} and for all $\epsilon \in [0,1]$:
    \begin{equation}
    \prob^{N_v}_{\uniform[\decisionspace]}
    \left[\prob_{\uniform[\decisionspace]}\left[\variance(\decisionvariable) \leq \variance^*_{N_v}\right] \geq 1-\epsilon \right] \geq 1-(1-\epsilon)^{N_v}.
    \end{equation}
    The set in the interior probability corresponds to the level set of $\variance^*_{N_v}$, \textit{i.e.}
    \begin{equation}
    \prob^{N_v}_{\uniform[\decisionspace]}
    \left[\prob_{\uniform[\decisionspace]}\left[\Omega_{\variance^*_{N_v}}\right] \geq 1-\epsilon \right] \geq 1-(1-\epsilon)^{N_v}.
    \end{equation}
    Finally, via~\eqref{eq:min_prob}, we know that the probability of sampling a decision in the level set corresponding to the optimality gap is $1-p$.  Restricting to $\epsilon \in [0,p]$ which implies that $1-\epsilon \geq 1-p$ and substituting terms in the inequality above, we have the following result.
    \begin{equation}
    \prob^{N_v}_{\uniform[\decisionspace]}
    \left[\prob_{\uniform[\decisionspace]}\left[\Omega_{\variance^*_{N_v}}\right] \geq \prob_{\uniform[\decisionspace]}\left[\Omega_{\optimalitygap(\decisionvariable^*_{N_p})}\right] \right] \geq 1-(1-\epsilon)^{N_v}.
    \end{equation}
    Then, the final result holds due to Lemma~\ref{lem:equivalencies}.
\end{proof}

In summary, Theorem~\ref{thm:bound_gap} tells us that if we wish to bound the optimality gap of a percentile solution, we need to evaluate the variance of $N_v$ uniform samples $\decisionvariable$ from $\decisionspace$ with respect to a subset of the information set $I_1$ utilized to generate the percentile solution.  In practice, however, the exact probability $p$ of sampling decisions with large enough variance will be unknown to the practitioner \textit{apriori}.  In these cases, it suffices to assume a small enough value for $\epsilon$, \textit{i.e.} $10^{-2}$ or smaller, is smaller than $p$.  Examples along this vein will be provided in the following section.  Notably, this result implies that we can utilize percentile methods to both identify decisions that outperform a large fraction of the decision space and also determine their optimality gap.  Indeed, this result holds even for non-convex optimization problems, provided they satisfy Assumption~\ref{assump:boundedness}.

\subsection{Producing Solutions with Maximum Optimality Gaps}
\label{sec:repetitive_gap}
The prior section provides a method to determine the upper bound on the optimality gap of a provided solution via a secondary sampling scheme.  This section provides a method to remove the secondary sampling requirement for similar optimization problems to be solved successively.  In other words, consider the following general form of optimization problems, where each instance $l$ satisfies Assumption~\ref{assump:boundedness}:
\begin{equation}
\label{eq:repetitive_opt}
J^{l*} = \min_{\decisionvariable \in \decisionspace^l}~J^l(\decisionvariable),~(\decisionspace^l,J^l) \in \optimalityset,~l \in L.
\end{equation}
Furthermore, we assume that it is possible to randomly sample indices $l$ from $L$, \textit{e.g.} via the uniform distribution.

\newidea{Example Setting:} Here, $\optimalityset$ is a set containing pairs of objective functions and decision spaces.  To provide an example consistent with the sections to follow, consider the following nonlinear dynamical system with state $x$ and input $u$:
\begin{equation}
    \label{eq:nonlinear_model}
    x_{k+1} = f(x_k,u_k),~x \in \mathcal{X},~u \in \mathcal{U}
\end{equation}
Provided a cost function over states and inputs, state constraints, and torque bounds, we can construct the following finite-time optimal controller with horizon $H$ for the aforementioned system.  Here, all state constraints are projected to input constraints through prediction over the model~\eqref{eq:nonlinear_model}:
\begin{align}
\label{eq:general_FTOCP} \tag{FTOCP}
& \argmin_{\mathbf{u}=(u^0,u^1,\dots,u^{H-1}) \in \mathcal{U}^H}~& & J(\mathbf{u}, x_k), \\
& ~\qquad \mathrm{subject~to~} & & \mathbf{u} \in \mathbb{U}(x_k) \subseteq \mathcal{U}^H.
\end{align}
The aforementioned finite-time optimal controller~\eqref{eq:general_FTOCP} collapses to the form in~\eqref{eq:repetitive_opt} if we consider an optimality set $\optimalityset$ indexed by states $x \in \mathcal{X}$ --- a specific form of indexing more generally referred to via $l \in L$ in~\eqref{eq:repetitive_opt}.

\newidea{Key Insight:} The critical insight for this section then is as follows. If we were to randomly sample via a distribution $\pi$ over $\optimalityset$, optimization problems of the form in~\eqref{eq:repetitive_opt} and calculate the optimality gap $\optimalitygap^l$ of percentile solutions for that problem, the corresponding gap is a sample of some real-valued random variable.  By taking multiple independent samples of this random variable, we can provide, using recent results on risk-measure estimation, a probabilistic upper bound on this random variable, \textit{i.e.} a probabilistic upper bound on achievable optimality gaps.  Phrased formally, an existing theorem in the literature we will utilize is as follows.
\begin{theorem}
    \label{thm:upper_bound_var}
    Let $X$ be a real-valued random variable with (perhaps) unknown distribution $\pi$, let $\{x_i\}_{i=1}^N$ be a set of $N$ independent samples of $X$, and let $x^*_N$ be the maximum such sample.  $\forall~\epsilon \in [0,1]$, the following statement holds:
    \begin{equation}
        \prob^N_{\pi}\left[\prob_{\pi}[x \leq x^*_N] \geq 1-\epsilon \right] \geq 1-(1-\epsilon)^N.
    \end{equation}
\end{theorem}
\begin{proof}
    This follows from a scenario argument, see~\cite{campi2008exact}, and is a re-phrasing of Theorem~4 in~\cite{akella2022sample}.
\end{proof}

To employ Theorem~\ref{thm:upper_bound_var}, we require the following definition.
\begin{definition}
    \label{def:optimality_gap_RV}
    Let $\gapRV(N_p)$ be a real-valued random variable with distribution $\pi_{\optimalitygap}$ and samples $\gapRVsample$ defined as follows: (1) Uniformly sample an index $l \in L$, (2) Take a percentile approach to solve~\eqref{eq:repetitive_opt} corresponding to this sampled index $l$, producing the solution $\decisionvariable^{l*}_{N_p}$, (3) Calculate and report as a sample $\gapRVsample$, the optimality gap $\optimalitygap^l\left(\decisionvariable^{l*}_{N_p}\right)$.
\end{definition}
\noindent Then, we can upper bound the optimality gaps of percentile solutions to all optimization problems formed in the set $\optimalityset$, to some minimum probability.  The formal statement will follow.
\begin{theorem}
    \label{thm:probabilistic_ub_optimalitygap}
    Let $\gapRV(N_p)$ be as per Definition~\ref{def:optimality_gap_RV}, and let $\{\gapRVsample_i\}_{i=1}^R$ be a set of $R$ independent samples of $\gapRV(N_p)$ with $\gapRVsample^*_R$ the maximum such sample.  Then $\forall~\epsilon \in [0,1]$, percentile solutions to optimization problems of the form in~\eqref{eq:repetitive_opt} will exhibit optimality gaps less than $\gapRVsample^*_R$ with minimum probability $1-\epsilon$ and confidence $1-(1-\epsilon)^R$, \textit{i.e.}
    \begin{equation}
        \prob^R_{\pi_{\optimalitygap}}\left[\prob_{\pi_{\optimalitygap}}[\gapRVsample \leq \gapRVsample^*_R] \geq 1-\epsilon \right] \geq 1-(1-\epsilon)^R.
    \end{equation}
\end{theorem}
\begin{proof}
    Stems via Theorem~\ref{thm:upper_bound_var}.
\end{proof}

In short, if we wish to remove the secondary sampling requirement for the determination of optimality gaps, we need to be able to calculate the optimality gap for at least $R$ independently chosen optimization problems of the form~\eqref{eq:repetitive_opt}.  Doing so permits us to make a general statement on the maximum achievable optimality gaps, to some minimum probability.  Now, we will provide a few examples.

\section{Examples}
\label{sec:examples}
\begin{figure}[t]
    \centering
    \includegraphics[width = \columnwidth]{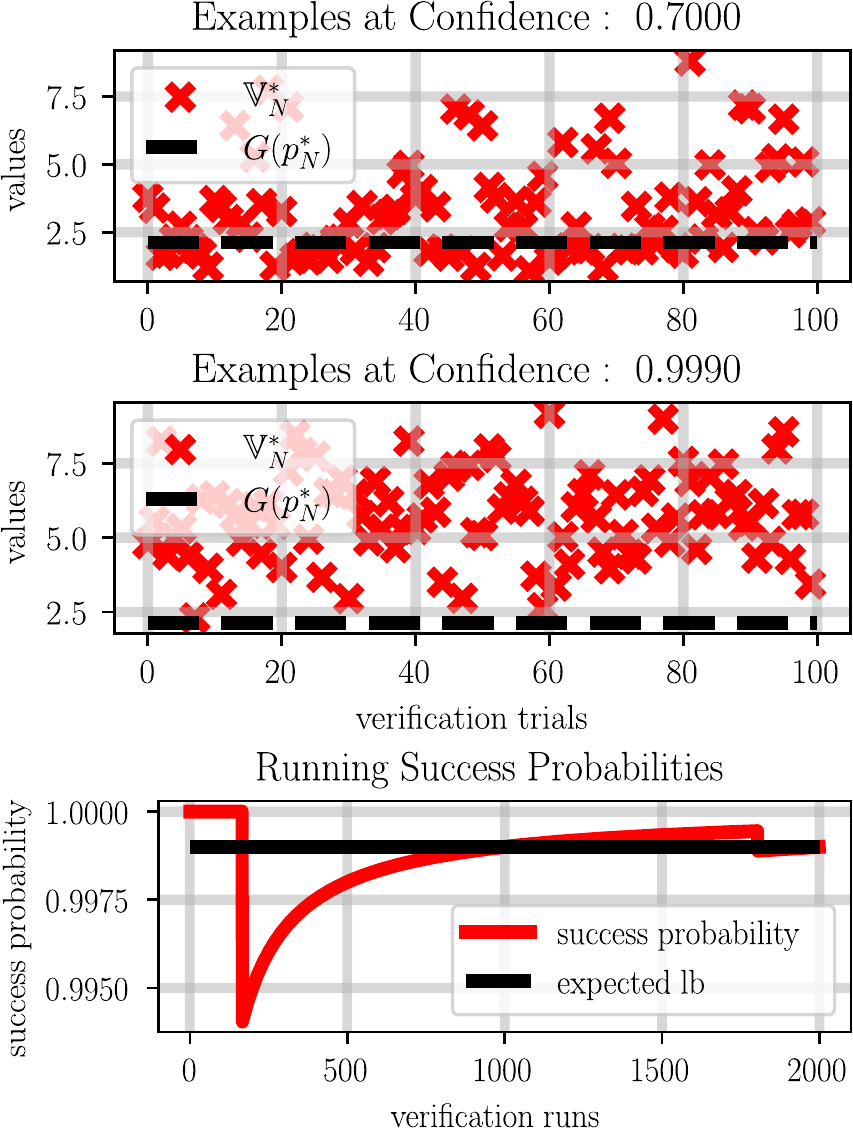}
    \caption{Validation Data for Section~\ref{sec:tsp} corresponding to Theorem~\ref{thm:bound_gap}.  (Top) $100$ reported upper bounds $\variance^*_{N_v}$ using Theorem~\ref{thm:bound_gap} with desired confidence for~\eqref{eq:prob_bound_gap} equal to $0.7$.  (Middle) $100$ reported upper bounds $\variance^*_{N_v}$ with confidence $0.999$.  (Bottom) Running fraction over $2000$ trials of reported upper bounds $\variance^*_{N_v}$ exceeding the true optimality gap $\optimalitygap(p^*_{N_p})$ at confidence level $0.999$.  Notice how the fraction of upper bounds exceeding the optimality gap increases as we increase confidence (top to middle), and the running fraction of upper bounds exceeding the optimality gap converges to our desired confidence (bottom), corroborating Theorem~\ref{thm:bound_gap}.}    \vspace{-0.2 in}
    \label{fig:tsp_examples}
\end{figure}
\subsection{Validating Theorem 2 --- Traveling Salesman}
\label{sec:tsp}
The Traveling Salesman Problem (TSP) is a classic non-convex path planning problem referencing the identification of the path of the shortest distance traversing each node in a set once.  Mathematically, consider a set of waypoints $W = \{w_1,w_2,\dots,w_{|W|}\}~w_i \in \mathbb{R}^2$ and the set of all paths over these waypoints $P = \{(i_1,i_2,\dots,i_{|W|})~|~i_j \in \{1,2,\dots,|W|\},~i_j \neq i_k,~\forall~j\neq k\}$.
Then formally, the Traveling Salesman Problem is to
\begin{equation}
    \min_{p \in P}~\sum_{i=0}^{|W|-1}~\|p_i - p_{i+1}\| + \|p_0 - p_{|W|}\|.
\end{equation}
For a graph with $10$ nodes and $3628800$ possible paths, evaluating $N_p = 5000$ paths and picking the best one identifies a path $p^*_{N_p}$ in the $99.9\%$-ile with at least $99\%$ confidence according to Theorem~\ref{thm:prob_optimality}.  Using a subset $D$ of the corresponding information set $I_1$ for the determination of such a percentile solution (see Assumption~\ref{assump:percentile} for the definition of $I_1$), we define a variance function $\variance$ as per equation~\eqref{eq:variance}.  Finally, to validate the probabilistic results of Theorem~\ref{thm:bound_gap}, we can also calculate the true probability $p$ of uniformly sampling paths that exhibit a variance higher than the optimality gap of our proposed solution $\optimalitygap(p^*_{N_p})$ --- this is the minimum probability assumed to exist via Assumption~\ref{assump:fairness} and defined in equation~\eqref{eq:min_prob}.  For this particular node set and percentile solution $p^*_{N_p}$ the probability $p = 0.1083$.

\begin{table}[t]
    \centering
    \begin{tabular}{|p{0.75cm}|p{0.5cm}|p{0.5cm}|p{1.5cm}|p{1.5cm}|p{1.5cm}|}
    \hline
    \tabcenter{Name} & \tabcenter{$N_p$} & \tabcenter{$N_v$} & \tabcenter{expected \\ success \\ probability~\eqref{eq:prob_bound_gap}} & \tabcenter{true success \\ probability} & \tabcenter{average \\ runtime (ms)} \\
    \hline
    R-2 & $300$ & $300$ & $0.95$ & $\approx 1$ & 5.18\\
    \hline
    R-10 & $300$ & $300$ & $0.95$ & $\approx 1$ & 5.17\\
    \hline
    Ack & $300$ & $300$ & $0.95$ & $\approx 1$ & 5.28\\
    \hline
    Ble & $300$ & $300$ & $0.95$ & $\approx 1$ & 5.17\\
    \hline
    Levi & $300$ & $300$ & $0.95$ & $\approx 1$ & 5.31\\
    \hline
    Himm & $300$ & $300$ & $0.95$ & $\approx 1$ & 5.30\\
    \hline
    \end{tabular}
    \caption{Data for Section~\ref{sec:success_probabilities} for the (R-2) Rastigrin 2-D, (R-10) Rastigrin 10-D, (Ack) Ackley, (Ble) Beale, Levi, and (Himm) Himmelblau benchmark optimization problems.} \vspace{-0.3 in}
    \label{tab:benchmark data}
\end{table}

\begin{figure*}
    \centering
    \includegraphics{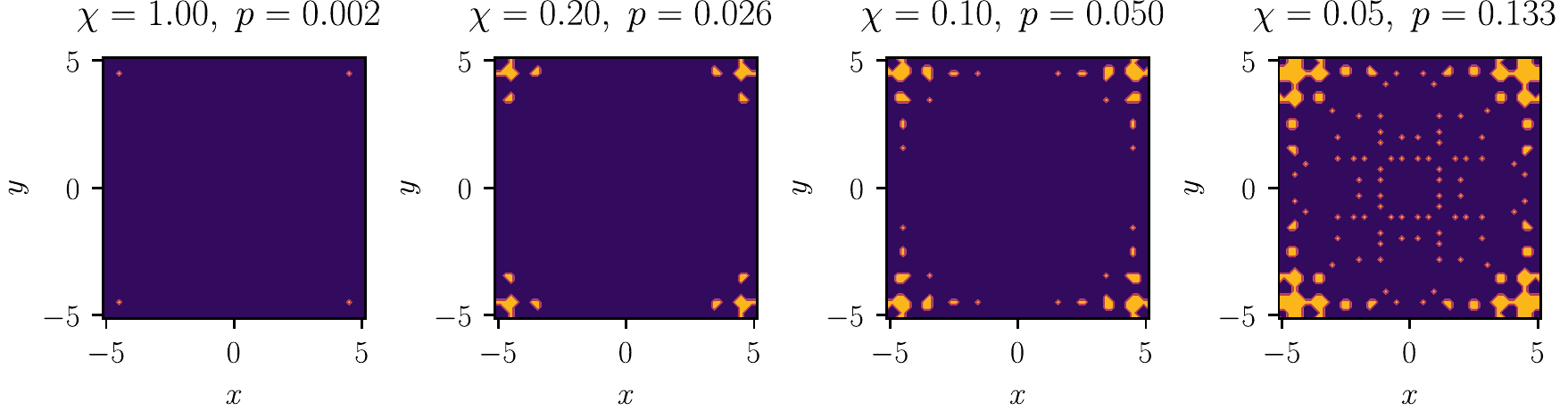}
    \caption{Validation data for Section~\ref{sec:success_probabilities}.  We claim that by varying the amount of information used to generate the variance function $\variance$ in equation~\eqref{eq:variance}, we can change the baseline probability $p$ of sampling a decision whose variance exceeds the optimality gap of a given solution $\optimalitygap(s^*_N)$ (such decisions are highlighted in orange).  Notice that as the volume fractions $\chi$ occupied by the chosen information set $D$ decreases, we see a corresponding increase in the baseline probability $p$.  Section~\ref{sec:success_probabilities} discusses why this inverse relationship holds.}
    \label{fig:changing_V}
\end{figure*}

\newidea{Figure Analysis:} To validate the results of Theorem~\ref{thm:bound_gap}, we solved for the minimum number of samples $N_v$ required to determine an upper bound $\variance^*_{N_v}$ to $\optimalitygap(p^*_{N_p})$ with minimum confidence $0.7$ --- top figure in Figure~\ref{fig:tsp_examples} requiring $11$ samples--- and minimum confidence $0.999$ --- middle figure in Figure~\ref{fig:tsp_examples} requiring $61$ samples.  Both of these minimum sample requirements were identified by setting $\epsilon = p$ in~\eqref{eq:prob_bound_gap} and solving for the minimum integer $N_v$ required to make the right-hand side greater than or equal to our desired confidence.  For each confidence level, we performed $100$ separate trials and reported the maximum variance $\variance^*_{N_v}$ produced by each trial according to Theorem~\ref{thm:bound_gap}.  As can be seen in the corresponding data, increasing the confidence increases the likelihood that the corresponding reported result $\variance^*_{N_v}$ exceeds the true optimality gap $\optimalitygap(p^*_{N_p})$ --- the red $x$'es in the middle figure are all above the black, dashed line, whereas a few dip below the same line in the top figure when we report solutions with lower confidence.  Furthermore, by repeating the procedure once more at confidence $0.999$, taking $2000$ separate verification runs, and recording whether $\variance^*_{N_v} \geq \optimalitygap(p^*_{N_p})$ per run, we can get a sense of the true, running probability that $\variance^*_{N_v} \geq \optimalitygap(p^*_{N_p})$.  As can be seen in the bottom figure, this probability converges to $0.999$ --- the lower bound expected by Theorem~\ref{thm:bound_gap}.  Notably, though, this result implies that we were able to identify a path in the $99\%$-ile that was no more than $1.12$ times the length of the optimal path, by only evaluating $5061$ uniformly chosen path samples, less than $0.14\%$ of the overall decision space.

\subsection{Increasing Success Probabilities --- Benchmark Functions}
\label{sec:success_probabilities}
In a brief remark after defining the variance function in equation~\eqref{eq:variance}, we mentioned that by specific choice of a subset $D$, one could increase the baseline probability $p$ of uniformly choosing samples that exhibit a higher variance than the optimality gap of the reported percentile solution.  This section provides evidence in support of that statement for a few benchmark optimization problems.  The one referenced in Figure~\ref{fig:changing_V}, the 2-d Rastigrin function~\cite{rastrigin1974systems}, is as follows:
\begin{equation}
    \label{eq:rastigrin_example}
    \min_{x \in [-5.12,5.12]^2}~20+\sum_{i=1}^2(x_i^2 -10\cos(2\pi x_i)).
\end{equation}
Both the decision space and objective function are bounded, permitting a percentile solution to~\eqref{eq:rastigrin_example}.  Following Theorem~\ref{thm:prob_optimality} and taking $N_p=100$ samples for such an approach, we generate the information set $I_1$ and percentile solution $x^*_{N_p}$.  Then, by choice of a subset $D$ of $I_1$ for the definition of the variance function in~\eqref{eq:variance}, we claim we can vary the baseline probability $p$ of sampling decisions that exhibit a higher variance than the true optimality gap.  To show this, define $\chi$ to be the volume fraction the subset $D$ in~\eqref{eq:variance} occupies of $I_1$, \textit{i.e.} $\chi = 1$ implies $D = I_1$ and $\chi = 0.05$ implies that $D$ contains $1/20$-th as many elements as $I_1$.

\newidea{Figure Analysis:}  Figure~\ref{fig:changing_V} portrays the results of varying the volume fraction $\chi$ of decisions $D$ utilized to generate the variance function $\variance$ per equation~\eqref{eq:variance}.  The decisions highlighted in orange are those that exhibit a higher variance than the optimality gap of the reported solution.  Notice that as $\chi$ decreases, $p$ increases as indicated in the titles.  This inverse relationship arises as if we consider two sets $D_1 \subset D_2$ utilized for generation of the variance functions $\variance_1,\variance_2$, respectively, then $\forall~\decisionvariable \in \decisionspace$, $\variance_1(s) \geq \variance_2(s)$ by definition of $\variance$ as a minimization problem.  In other words, decreasing the amount of information provided to the variance function increases the corresponding conservativeness of the resulting function, which increases the likelihood of sampling a, now more conservative, upper bound on the optimality gap.  In practice, and in the examples provided in the prior section, using a dilation $\chi = 0.1$ proved most effective, though studying if there exists an optimal volume fraction remains an open problem and the subject of future work.

\newidea{Table Description:} By increasing the success probability $p$ of sampling decisions whose variance exceeds the optimality gap, we can ``blindly" use Theorem~\ref{thm:bound_gap} to bound the optimality gap of percentile solutions to benchmark optimization problems.  Table~\ref{tab:benchmark data} shows our data in this vein.  For each benchmark optimization problem, we produced a percentile solution using $N_p = 300$ samples, constructed a variance function $\variance$ using $10\%$ of the information set $I_1$ generated via the percentile method, and took $N_v = 300$ samples to identify the probabilistic maximum variance $\variance^*_{N_v}$.  Under the assumption that $p \geq 0.01$, Theorem~\ref{thm:prob_optimality} tells us that $\variance^*_{N_v} \geq \optimalitygap(x^*_{N_p})$ with $95\%$ probability --- column 4 in Table~\ref{tab:benchmark data}.  Indeed, over $5000$ trials following the above procedure for each optimization problem, we were successfully able to identify a valid upper bound every time.  Additionally, as the sampling method only requires the evaluation of sampled points, which is relatively quick, the procedure takes very little time to implement, as seen in the rightmost column.

\begin{figure}[t]
    \centering
    \includegraphics[width = \columnwidth]{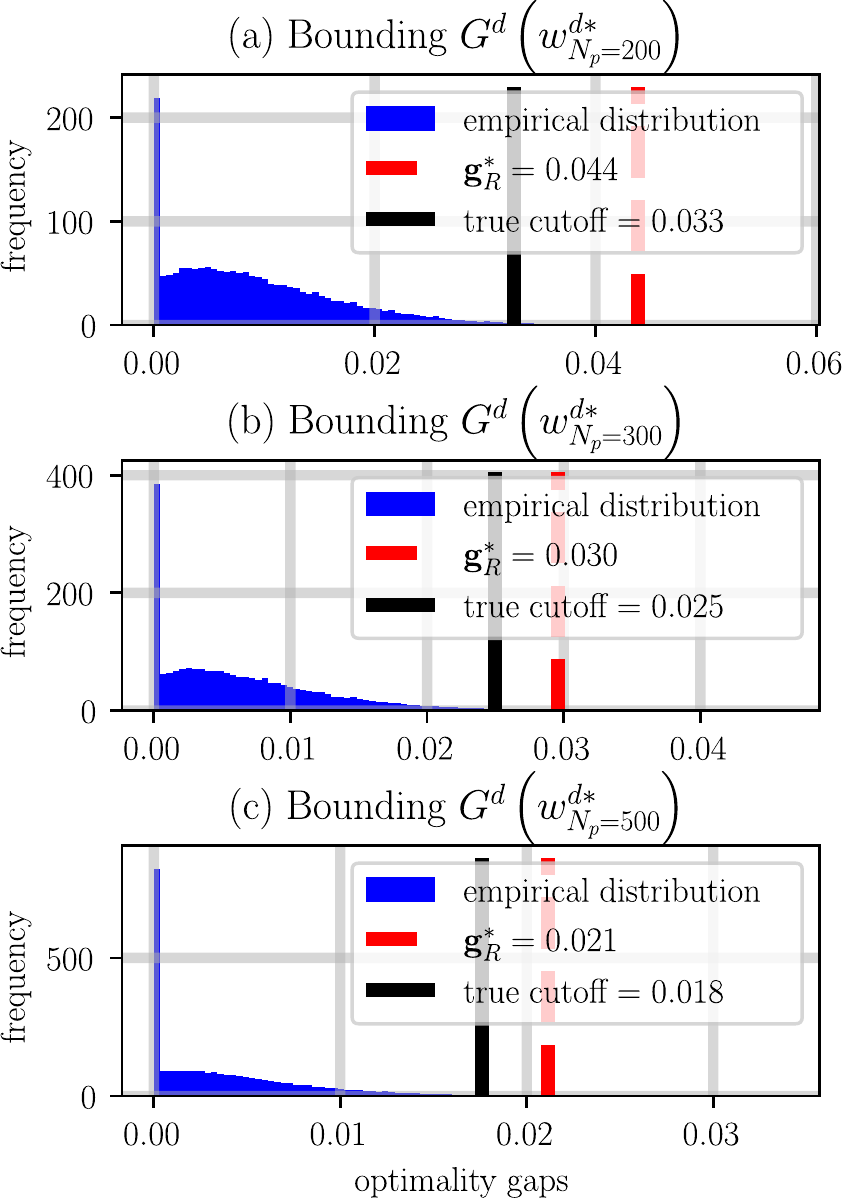}
    \caption{Validation data for Section~\ref{sec:robotarium} in support of Theorem~\ref{thm:probabilistic_ub_optimalitygap}.  We claim that we can upper bound the optimality gap of successive applications of percentile methods to solve optimization problems of the form in~\eqref{eq:repetitive_opt}.  Shown above in red are the calculated upper bounds for the black lines corresponding to the $99\%$ cutoff value of optimality gaps for percentile solutions to~\eqref{eq:augmented_NMPC}.  For the three separate percentile methods shown, we're able to upper bound the true value every time, corroborating Theorem~\ref{thm:probabilistic_ub_optimalitygap}.  This information is further discussed in Section~\ref{sec:robotarium}.}
    \vspace{-0.2 in}
    \label{fig:mpc_gap_info}
\end{figure}

\subsection{Validating Theorem~\ref{thm:probabilistic_ub_optimalitygap} --- Nonlinear MPC}
\label{sec:robotarium}
To validate the results of Theorem~\ref{thm:probabilistic_ub_optimalitygap}, we require a series of optimization problems of the form in~\eqref{eq:repetitive_opt}.  Keeping with the example provided in Section~\ref{sec:repetitive_gap}, we aim to bound the optimality gap of percentile solutions for a Nonlinear MPC controller steering the Robotarium robots~\cite{wilson2020robotarium} --- a collection of agents modelable via unicycle dynamics.  Formally, let $x \in \mathcal{X} = [-1.6,1.6]\times[-1.2,1.2]\times[0,2\pi]$ be the system state, and let $u \in \mathcal{U} = [-0.2,0.2] \times [-\pi,\pi]$ be the control input.  Then, our model is as follows:
\begin{equation}
    \label{eq:unicyle_model}
    \begin{aligned}
        x_{k+1} & = \underbrace{x_k +
        \begin{bmatrix}
        \cos\left(x_k[3]\right) & 0 \\
        \sin\left(x_k[3]\right) & 0 \\
        0 & 1
        \end{bmatrix}u_k(\Delta t = 0.033)}_{f(x_k,u_k)}.
    \end{aligned}
\end{equation}
Furthermore, each agent has a Lyapunov controller $U$ steering it to a provided waypoint $w \in \mathcal{W}$:
\begin{gather}
    U: \mathcal{X} \times \mathcal{W} \triangleq [-1.6,1.6] \times [-1.2,1.2] \to \mathcal{U}.
\end{gather}

We will use $U$ to construct our MPC algorithm, which provides waypoints to steer the system around static and moving obstacles toward at least one goal.  Formally, we overlay an $8\times5$ grid on $\mathcal{W}$ and define the space of operating environments $\mathcal{D}$ as those environments that: (1) have $8$ static obstacles (SO) and $3$ goals ($g$), (2) have controlled ($x_a$) and uncontrolled ($x_o$) agent starting positions outside static obstacles and goals, and (3) have at least one feasible path from the controlled agent's starting location to one of the three goals.  A vector $d = [x_a,x_o,\mathrm{SO},g] \in \mathcal{D}$ corresponds to one such environmental setup. To account for collisions, consider the following barrier function $h$, where $P = [\mathbf{I}_{2 \times 2}~\mathbf{0}_{2 \times 1}]$ projects system states to the plane~\cite{ames2016control}:
\begin{equation}
    \label{eq:robotarium_barrier}
    h(x_a,x_o,d) = \begin{cases}
        -5 & \mbox{in~SO~cell}, \\
        \|P(x_a - x_o)\| - 0.18 & \mbox{else}.
    \end{cases}
\end{equation}

Then the nominal NMPC algorithm minimizes $S:\mathcal{W} \to \mathbb{R}$ --- a function outputting the shortest path distance from a waypoint to the closest goal --- while ensuring that the existing Lyapunov controller $U$ renders the barrier $h$ positive for the next $5$ time steps, \textit{i.e.}, $\forall~j=1,2,\dots,5$, 
\begin{align}
w^*_k = &~\argmin_{w \in \mathcal{W}}~& & S(w), \label{eq:rob_NMPC} \tag{NMPC-A} \\
& \mathrm{subject~to~} & & x^{j}_{k} = f(x^{j-1}_k,u^{j-1}), \label{eq:constr_1} \tag{a}\\
& & & x^0_k = x_k, \label{eq:constr_2} \tag{b} {\color{white} \eqref{eq:constr_2},\eqref{eq:constr_3}}\\
& & & h(x^j_{k,a},x_o,d) \geq 0 \label{eq:constr_3} \tag{c} \\
& & & u^{j-1} = U\left(x^{j-1}_k, w\right), \label{eq:constr_4} \tag{d} \\
& & & 0.05 \leq \|w - x_k\| \leq 0.2.
\end{align}
To ease sampling requirements, we consider an augmented cost $J$ that outputs $100$ whenever a waypoint $w$ fails to satisfy constraints~\eqref{eq:constr_1}-\eqref{eq:constr_4} in~\eqref{eq:rob_NMPC}, yielding the following:
\begin{align}
w^{d*} = &~\argmin_{w \in \mathcal{W}}~& & J(w,d),\label{eq:augmented_NMPC} \tag{NMPC-B} \\
& \mathrm{subject~to~} & & 0.05 \leq \|w - x\| \leq 0.2.
\end{align}
Finally, we note that~\eqref{eq:augmented_NMPC} is equivalent to~\eqref{eq:repetitive_opt}, if we consider as index set $L$, the environment set $\mathcal{D}$.

\newidea{Figure Analysis:} To validate the results of Theorem~\ref{thm:probabilistic_ub_optimalitygap} we uniformly sample $R = 459$ different environments $d \in \mathcal{D}$ and calculate a percentile solution to~\eqref{eq:augmented_NMPC} with $N_p=200$ samples, Figure~\ref{fig:mpc_gap_info}-(a); $N_p=300$ samples, Figure~\ref{fig:mpc_gap_info}-(b); and $N_p=500$ samples, Figure~\ref{fig:mpc_gap_info}-(c).  We calculate the optimality gap for each solution by performing gradient descent on the best out of $2000$ uniformly chosen samples and reporting the final value as the true optimal value.  According to Theorem~\ref{thm:probabilistic_ub_optimalitygap}, in each case, we should produce an upper bound on optimality gaps $\gapRVsample^*_R$ that exceeds sample-able optimality gaps $\gapRVsample$ with at least $99\%$ probability and $99\%$ confidence.  To validate this statement, we uniformly sampled $50000$ more environments $d \in \mathcal{D}$ and followed the prior scheme for each percentile case to determine the distribution of sample-able optimality gaps.  This data is reflected as the distributional data you see in each subfigure in Figure~\ref{fig:mpc_gap_info}.  As can be seen, in each case the reported upper bound $\gapRVsample^*_R$ exceeds the true $99\%$ cutoff.  Furthermore, as the number of samples taken for the percentile solution increases, the upper bound decreases.  This is expected as we are providing a solution in a higher percentile each time.  To emphasize the utility of this result for controls, say we wished to implement a percentile procedure with $N_p=300$ samples to provide ``good" waypoints optimizing for~\eqref{eq:augmented_NMPC}.  Offline calculation of this optimality gap would provide confidence that in practice, we would, with $99\%$ probability and with $99\%$ confidence, be choosing waypoints within $0.03m$ of the optimal waypoint at every iteration.  Notably, we would not have to solve the non-convex program and repetitively sample variances at each time step to make this statement.

\section{Conclusion}
We provide a method to bound the optimality gaps of percentile methods to a wide class of non-convex optimization problems and explore the utility of such approaches for control in an NMPC setting.  We also verify our mathematical results on several benchmark optimization problems including the traveling salesman problem.  In future work, we hope to decrease the conservatism of our bounds and formally prove the existence of a baseline probability for certain classes of optimization problems, \textit{i.e.} remove the requirement on our fairness assumption preceding our theoretical statements.

\bibliographystyle{IEEEtran}
\bibliography{IEEEabrv,bib_works}

\end{document}